\documentclass[12pt,reqno]{amsart}
\usepackage{amsfonts}
\usepackage{amsthm}
\usepackage{amsmath}
\usepackage{amscd}
\numberwithin{equation}{section}
\usepackage{hyperref}
\hypersetup{nesting=true,debug=true,naturalnames=true}
\usepackage{graphicx,amssymb,upref}
\usepackage{enumerate}
\usepackage[all]{xy}
\usepackage{color}
\usepackage{mathrsfs}
\usepackage{caption}
\usepackage{paralist}
\usepackage{color}
\usepackage{float}
\usepackage{tikz,pgf}
\usetikzlibrary{calc}
\usetikzlibrary{intersections,decorations.markings}
\usepackage{tikz-cd}
\usepackage{lineno}

\theoremstyle{definition}
\newtheorem{theorem}{\bf Theorem}[section]

\newtheorem{lemma}[theorem]{\bf Lemma}

\theoremstyle{definition}

\newtheorem{question}[theorem]{\bf Question}
\newtheorem*{theorem1}{\bf Theorem A}
{

}
\newcommand{\mm}[1]{\mathrm{#1}}
\newcommand{\mb}[1]{\mathbb{#1}}
\newcommand{\mc}[1]{\mathcal{#1}}
\makeatletter
\@namedef{subjclassname@2020}{\textup{2020} Mathematics Subject Classification}
\makeatother

\begin{document}

\title[Non-negative curvature on certain product manifolds]{Non-negative curvature on certain product manifolds}

\author[Wen Shen]{Wen Shen}
\email{shenwen121212@163.com}

\begin{abstract}
Let $G/H$ be a closed, simply connected homogeneous manifold. Suppose every stable class of real vector bundles over $G/H$ contains a homogeneous bundle. 
Then, for any closed, simply connected smooth manifold $M$ homotopy equivalent to $G/H$, there exists $n>\mm{dim}(M)$ such that the product manifold $M\times S^{n}$ admits a metric with non-negative sectional curvature. Many homogeneous manifolds satisfy this assumption, including simply connected compact rank-one symmetric spaces, and among others.
\end{abstract}


\maketitle

\section{Introduction}\label{intro}

In this paper, all manifolds are assumed to be smooth.

Since the inception of global Riemannian geometry, a central question has been whether a given manifold admits a Riemannian metric with positive sectional curvature, non-negative sectional curvature, positive Ricci curvature, positive scalar curvature, or other curvature properties. Metrics with curvature bounds play a pivotal role in elucidating the intrinsic geometric and topological characteristics of manifolds, and their existence (or non-existence) often serves as a key indicator of the manifold's underlying structure.

In the realm of simply connected manifolds of dimension 
$n\ge 5$, significant progress has been made concerning obstructions to the existence of metrics with positive scalar curvature \cite{GL, Stolz}.
However, there is currently no analogous general result that comprehensively characterizes when a manifold can admit a non-negatively curved metric. To date, the only known obstruction to admitting a metric of non-negative sectional curvature is Gromov's bound \cite{Gromov} on the total Betti number. This theoretical gap motivates our investigation into the existence of metrics with non-negative sectional curvature on specific classes of manifolds, with the aim of contributing to a more complete understanding of the intricate interplay between the topology and geometry of manifolds.

From the results in \cite{GL,Stolz}, we observe an interesting phenomenon:
 for any closed, simply connected $m$-manifold $M$, there exists an integer $n>m$ such that the product manifold $M\times S^n$ admits a metric with positive scalar curvature.  
 This motivates the following question:
\begin{question}\label{quemain}
For any closed, simply connected $m$-manifold $M$, does there exist an integer $n>m$ such that the product  $M\times S^n$ admits a metric with non-negative sectional curvature (\(\mathrm{sec} \geq 0\))?
\end{question}

In \cite{Shen1}, it is shown that for any closed 6-manifold $M$ homotopy equivalent to $\mb{CP}^3$, if its first Pontrjagin class satisfies $\mathfrak p_1(M)\ge 4$, then $M\times S^7$ is diffeomorphic to the total space of the sphere bundle of a real rank-$8$ vector bundle over $\mb{CP}^3$. Furthermore, this total space admits a metric with \(\mathrm{sec} \geq 0\) \cite{Shen}. So $M\times S^7$ does as well. 

Note that $\mb{CP}^3$ is a compact rank-one symmetric space. Let $\xi$ be an arbitrary real vector bundle over a compact rank-one symmetric space. Denote by $k\in \mb{N}$ the trivial real vector bundle of rank $k$. By \cite{GAD2017}, for some $k$, the total space of the Whitney sum $\xi\oplus k$ admits a metric with \(\mathrm{sec} \geq 0\). In particular, the total space of the sphere bundle of $\xi\oplus k$ also admits a metric with \(\mathrm{sec} \geq 0\). This leads one to ask whether Question \ref{quemain} holds for closed, simply connected manifolds homotopy equivalent to compact rank-one symmetric spaces.

In this paper, we will show that Question \ref{quemain} holds for many closed, simply connected manifolds-including those homotopy equivalent to compact rank-one symmetric spaces. To present the main theorem of this paper, we first make some preliminary definitions.

Let $G$ be a compact Lie group, $H$ a closed subgroup of $G$. Let $\mm{Vect}_{\mb{R}}(G/H)$ denote the set of isomorphism classes of real vector bundles over $G/H$, and let $\mm{Rep}_\mb{R}(H)$ denote the set of isomorphism classes of real representations of $H$. Both $\mm{Vect}_{\mb{R}}(G/H)$ and $\mm{Rep}_\mb{R}(H)$ are semirings, and there exists a 
semiring homomorphism $\alpha:\mm{Rep}_\mb{R}(H)\to \mm{Vect}_{\mb{R}}(G/H)$. A real vector bundle over $G/H$
 is called \textbf{homogeneous} if its isomorphism class lies in the image of $\alpha$.

Two real vector bundles $\xi,\eta$ over $G/H$ are stably equivalent if there are trivial bundles $m_1,m_2\in \mb{N}$ such that $\xi\oplus m_1\cong\eta\oplus m_2$.  Let $S_{\mb{R}}(G/H)$ denote the set of stable classes of real vector bundles over $G/H$.

\begin{theorem1}
Let $G$ be a compact Lie group, $H$ a closed subgroup of $G$, $G/H$ a closed, simply connected manifold. Suppose every stable class in $S_{\mb{R}}(G/H)$ contains a homogeneous bundle, then for any closed manifold $M$ homotopy equivalent to $G/H$, there exists $n>\mm{dim}(M)$ such that $M\times S^n$ admits a metric with non-negative sectional curvature.	
\end{theorem1}

There exist many homogeneous spaces satisfying all conditions in Theorem A, such as 
 any simply connected compact rank-one symmetric space \cite{GAD2017} (including $S^n$, $\mb{CP}^n$, $\mb{HP}^n$, and the Cayley plane $\mm{Ca}\mb{P}^2$); any simply connected compact homogeneous space of dimension at most seven \cite{GoZi2021};  the Wallach manifolds $W^{12}=\mm{Sp}(3)/\mm{Sp}(1)^3$, and $W^{24}=F_4/\mm{Spin}(8)$ \cite{GoZi2021}.

Let $G/H$ be a homogeneous space as in Theorem A, $m$ denote its dimension. If $m\ge 5$, by \cite[II.4.10]{Brow1972}, there is an exact sequence of sets
$$P_{m+1}\to \Phi(G/H)\to [G/H,\mm{F/O}]\to P_m$$
where $\Phi(G/H)$ is the set of diffeomorphism classes of closed manifolds homotopy equivalent to $G/H$; $\mm{F/O}$ is the fiber of the fibration $\alpha^{\mm{O}}_{\mm{F}}:\mm{BO}\to \mm{BF}$, which maps from the classifying space $\mm{BO}$ of stable real vector bundles to the classifying space $\mm{BF}$ of stable spherical fibrations \cite{Rudyak}; and    
\[
P_m = 
\begin{cases}
\mathbb{Z}, & m \equiv 0 \pmod{4} \\
\mathbb{Z}_2, & m \equiv  2 \pmod{4} \\
0, & m\equiv 1\pmod 2
\end{cases}
\]
In particular, the set $\Phi(\mb{CP}^n)$ $(n\ge 3)$ has beed studied in \cite{Sull1996}.

The paper is structured as follows:
  In Section \ref{vectorbundle}, we discuss metrics on the total spaces of vector bundles over homogeneous spaces. In Section \ref{constru}, we construct the product manifolds and the associated sphere bundles.
  In Section \ref{normal5}, we establish a normal bordism and prove the main theorem.

\section{Vector bundles over homogeneous spaces}\label{vectorbundle}

Let $\mm{Vect}_{\mb{R}}(M)$ denote the set of isomorphism classes of real vector bundles over a manifold $M$. This set carries a semiring structure via the Whitney sum $\oplus$ and the tensor product $\otimes$ of bundles. For the trivial real vector bundle of rank $m$, we write simply $m$ when no confusion arises.

\begin{lemma}\cite[Lemma 9.3.5]{AGP2002}\label{inversofvectorbundle}
	Let $\xi\in \mm{Vect}_{\mb{R}}(M)$ with $M$ compact. Then there exists $\eta \in \mm{Vect}_{\mb{R}}(M)$ such that $\xi\oplus \eta$ is  trivial.
\end{lemma}

From now on, assume $M$ is compact. Two bundles $\xi,\eta\in \mm{Vect}_{\mb{R}}(M)$ are stably equivalent if there exist trivial bundles $m_1,m_2$ with $\xi\oplus m_1\cong \eta\oplus m_2$. Let $S_{\mb{R}}(M)$ denote the set of stable classes of bundles over $M$, and let $\{\xi \}$ denote the stable class of $\xi$. By Lemma \ref{inversofvectorbundle}, the Whitney sum endows $S_{\mb{R}}(M)$ with the structure of an abelian group.


Assume $M$ is a homogeneous space, i.e., $M=G/H$, where $G$ is a compact Lie group and $H$ is a closed subgroup of $G$. Let $\mm{Rep}_\mb{R}(G)$ denote the set of isomorphism classes of real representations of $G$, which forms a semiring under the direct sum and tensor product of representations. We write $m$ for the trivial real representation of dimension $m$.
 For the inclusion $i : H \to  G$ of the closed subgroup $H$, let $i^\ast: \mm{Rep}_\mb{R}(G) \to  \mm{Rep}_\mb{R}(H)$ denote
the semiring homomorphism induced by restricting representations of $G$ to $H$.

 We now define a semiring homomorphism \cite{GAD2017}:
$$\alpha:\mm{Rep}_\mb{R}(H)\to \mm{Vect}_{\mb{R}}(G/H)$$
For each $\rho \in  \mm{Rep}_\mb{R}(H)$, consider the diagonal right action of $H$ on $G \times \mb{R}^m$ ($m=\mm{dim}(\rho)$) given by:
$$(G\times \mb{R}^m)\times H\to G\times \mb{R}^m$$
$$((g,x),h)\to (g\cdot h,\rho(h)^{-1}\circ x)$$
Here, $g\cdot h$ is the group multiplication in $G$; since $H$ is compact, $\rho(h)\in \mm{O}(m)$, and $\rho(h)^{-1}\circ x$ denotes the linear action on $x$.

The quotient space $E_\rho := (G \times \mb
{R}^m)/H$ is the total space of the vector bundle $\alpha(\rho)=:\{ \pi_\rho : E_\rho \to  G/H\}$, where $\pi_\rho$ is the natural projection map. Vector bundles constructed in this way are called homogeneous. 

Since $G$ is compact, it admits a biinvariant metric $\langle ,\rangle_{G}$ with nonnegative sectional curvature. Endow $G \times \mb{R}^m$ with the product of $\langle ,\rangle_{G}$ and the flat Euclidean metric on $\mb{R}^m$. By O'Neill's Theorem \cite{On}, $E_\rho$ inherits a quotient metric with non-negative curvature.

For each $\rho\in \mm{Rep}_\mb{R}(H)$, the sphere bundle of $\alpha(\rho)$ has the form
$$S^{m-1}\to S_\rho=(G\times S^{m-1})/H\to G/H$$
Endow $G \times S^{m-1}$ with the product of $\langle ,\rangle_{G}$ and the standard metric on $S^{m-1}$. By O'Neill's Theorem, $S_\rho$ inherits a quotient metric with non-negative curvature.

Suppose every stable class in $S_{\mb{R}}(G/H)$ contains a homogeneous bundle. Then, for any real vector bundle $\xi$ over $G/H$, there exist $\rho \in  \mm{Rep}_\mb{R}(H)$ and $n,m \in \mb{N}$ such that
$$\xi\oplus n\cong \alpha(\rho)\oplus m\cong \alpha(\rho\oplus m)$$
Thus, $\xi\oplus n$ is a homogeneous vector bundle,  leading to:
\begin{lemma}\label{spherebundle}
	Let $G$ be a compact Lie group and $H$ a closed subgroup. Suppose the stable class $\{\xi\}\in  S_\mb{R}(G/H)$ contains a homogeneous bundle. Then there exists $k \in  \mb{N}$ such that: 
	\begin{enumerate}
		\item the total space of the sphere bundle of $\xi \oplus k$ admits a metric with non-negative curvature;
		\item $\mm{rank}(\xi\oplus k)> \mm{dim}(G/H)+1$ and $\mm{rank}(\xi\oplus k)+\mm{dim}(G/H)\equiv 3\pmod 4$. 
	\end{enumerate}
\end{lemma}

\section{Basic constructions}\label{constru}

Let $M$ be a closed simply connected manifold homotopy equivalent to a closed homogeneous space $G/H$, where $G$ is compact and $H$ is a closed subgroup. Clearly, $\mm{dim}(M)=\mm{dim}(G/H)$. 

Let $h:M\to G/H$ be a homotopy equivalence, with $h^{-1}:G/H\to M$ its homotopy inverse. 
For the tangent bundle $\tau(M)$ of $M$, we have 
$$h^\ast\circ h^{-1\ast}(\tau(M))\cong \tau(M)$$
Denote the bundle $h^{-1\ast}(\tau(M))$ over $G/H$ by $\xi$.

By Lemma \ref{inversofvectorbundle}, there is a bundle $-\tau(G/H)$ over $G/H$ such that $\tau(G/H)\oplus (-\tau(G/H))$ is trivial. Define a bundle over $G/H$ by $$\eta=:\xi\oplus (-\tau(G/H))$$

\begin{lemma}\label{constrlemma}
	Assume $G/H$ satisfies all conditions in Theorem A. Then there exists $k \in \mb{N}$ such that the total space $N$ of the sphere bundle of $\eta \oplus k$ admits a metric with non-negative  curvature. Moreover, $\mm{rank}(\eta\oplus k)>\mm{dim}(G/H)+1$, and $\mm{rank}(\eta\oplus k)+\mm{dim}(G/H)\equiv 3\pmod 4$.
\end{lemma}
\begin{proof}
	This follows directly from Lemma \ref{spherebundle}.
\end{proof}

Let $n=\mm{rank}(\eta\oplus k)$.
Let $\pi:N\to G/H$ denote the bundle projection. We have $\mm{dim}(N)=\mm{dim}(G/H)+n-1$.

Next we consider the following bundle morphisms
\[
\xymatrix@C=1.3cm{
S^{n-1}\ar[r]\ar[d]&D^n\ar[r]\ar[d]&\mb{R}^n\ar[d]^-{}\\
N\ar[d]^-{\pi}\ar[r]^-{ i}&{W}\ar[d]^-{f}\ar[r] &  E\ar[d]^{p}\\
G/H\ar@{=}[r]&G/H\ar@{=}[r]&G/H
}
\] 
where $E$ is the total space of the bundle $\eta\oplus k$, ${W}$ is the total space of the disk bundle of $\eta\oplus k$. In particular,
$N$ is the boundary of ${W}$, ${W}$ is an embedded submanifold of $E$. By the Gysin sequence \cite[p.438]{Hatcher} and $n-1>\mm{dim}(G/H)$, the cohomology ring of $N$ satisfies
\begin{equation}
	H^\ast(N)\cong H^\ast(G/H\times S^{n-1}) \label{cohomoN}
\end{equation}


It is well known that the tangent bundle of $E$ satisfies 
$$\tau(E)\cong p^\ast(\tau(G/H)\oplus \eta \oplus k)$$
Since the dimensions of ${W}$ and $E$ are equal, we obtain 
\begin{equation}
\tau({W})\cong f^\ast(\tau(G/H)\oplus \eta\oplus k)\label{tauW}
\end{equation}
Since $ i:N\to W$ is an embedding, we have
\begin{equation}
	\tau({N})\oplus 1\cong {i}^\ast(\tau(W)) \label{tauN}
\end{equation}


Finally we consider the product manifold $M\times S^{n-1}$. Clearly, 
\begin{equation}
	H^\ast(M\times S^{n-1})\cong H^\ast(G/H\times S^{n-1})\cong H^\ast(N) \label{cohomoMSm-1}
\end{equation}
and $\mm{dim}(M\times S^{n-1})=\mm{dim}(N)$. Furthermore, 
$M\times S^{n-1}$ is the boundary of $M\times D^{n}$ whose tangent bundle is the product of $\tau(M)$ and $\tau(D^n)$ where $\tau(D^n)$ is a trivial bundle. 

Consider the composition
$$r:M\stackrel{\Delta}{\to}M\times M\stackrel{id\times c}{\longrightarrow}M\times D^n$$
where $\Delta$ is the diagonal map, $c$ is a constant map. It is easy to verify that $r$ induces isomorphisms on $H_i$ for all $i\ge 0$. Note that both $M$ and $M\times D^n$ are CW complexes. Since $M$ is simply connected, $r$ is a homotopy equivalence. Moreover, we have
$$
	r^\ast(\tau(M\times D^n))\cong \tau(M)\oplus n$$
Let $r^{-1}$ be the homotopy inverse of $r$. Then, \begin{equation}
	r^{-1\ast}(\tau(M)\oplus n)\cong \tau(M\times D^n)\label{rhomoto}
\end{equation} 

For the composition $$\mathfrak f:M\times D^n\stackrel{r^{-1}}{\longrightarrow } M\stackrel{h}{\to }G/H$$
 the tangent bundle of $M\times D^{n}$ satisfies
\begin{equation}
	\tau(M\times D^{n})\cong\mathfrak f^\ast\xi \oplus n \label{tauMtimesD}
\end{equation}

\section{Normal bordism}\label{normal5}

For a manifold $X$, a normal bundle of $X$ is a bundle $\nu(X)$  such that $\tau(X)\oplus \nu(X)$ is trivial. All normal bundles of $X$ lie in the same stable class in $S_{\mb{R}}(X)$.



Recall the bundle $\xi$ over $G/H$ (see Section \ref{constru}). By Lemma \ref{inversofvectorbundle}, there exists a bundle $-\xi$ such that $\xi\oplus (-\xi)$ is trivial.
 

Recall the maps $f:W\to G/H$ and $i:N\to W$ from Section \ref{constru}.

\begin{lemma}\label{normalN}
(1)$f^\ast(-\xi)$ is a normal bundle of $W$.

(2) $(f\circ i)^\ast(-\xi)$ is a normal bundle of $N$.
\end{lemma}
\begin{proof}
Recall $\tau(G/H)\oplus \eta\cong \xi\oplus s$ for some $s\in \mb{N}$. The first statement follows from
	 \eqref{tauW}. 
	
	By \eqref{tauN}, $\tau(N)\oplus (h\circ i)^\ast(-\xi) \oplus 1$ is trivial. Since $\mm{rank}(\tau(N)\oplus (h\circ i)^\ast(-\xi) )>\mm{dim}(N)$, \cite[Lemma 3.5]{KervaMilnor} implies $\tau(N)\oplus (h\circ i)^\ast(-\xi)$ is trivial. This completes the proof.
\end{proof}

Recall the map $\mathfrak f:M\times D^n\to G/H$ from Section \ref{constru}. Let $\mathfrak i:M\times S^{n-1}\to M\times D^n$ be the natural embedding.
\begin{lemma}\label{normalMtimesS}
(1) $\mathfrak f^\ast(-\xi)$ is a normal bundle of $M\times D^n$.
 
 (2) $(\mathfrak f\circ \mathfrak i)^\ast(-\xi)$ is a normal bundle of $M\times S^{n-1}$.
\end{lemma}
\begin{proof}
	By \eqref{tauMtimesD}, the first statement follows. 

Note that $	\tau({M}\times S^{n-1})\oplus 1\cong \mathfrak{i}^\ast(\tau(M\times D^n))$.
	The second statement follows by the same argument as in the proof of Lemma \ref{normalN}.
\end{proof}

 Next we introduce a key theorem:
 \begin{theorem}\cite{Freed}\label{freed}
 Let $B$ be a simply connected, finite simplicial complex with a vector bundle $\gamma$. Let $g_i:K^{2n}_i\to B$ (for $i=1,2$ and $n\ge 3$) be normal maps from closed manifolds, i.e., $g^\ast_i(\gamma)=\nu(K_i)$. Suppose $n$ is odd, $g_1$ and $g_2$ are normally bordant, $g_i$ is $n$-connected, and $B_n(K_1) = B_n(K_2)$. Then $K_1$ is diffeomorphic to $K_2$.  
 \end{theorem}
 
 In the above theorem, $g_i$ is $n$-connected if it induces isomorphisms on $\pi_k$ for $k<n$, and an epimorphism on $\pi_n$.
 $B_n$ denotes the $n$-th Betti number. From the proof of Theorem \ref{freed} in \cite{Freed}, $g_1$ and $g_2$ are normally bordant if there exists a connected manifold $\mathcal W$ satisfies:
 \begin{enumerate}
 	\item Its boundary is the disjoint union of $K_1$ and $-K_2$ where $-K_2$ is $K_2$ with the opposite orientation. 
 	\item There exists a map $g:\mc{W}\to B$ such that $g^\ast\gamma$ is a normal bundle of $\mc{W}$, and the restriction $g|_{K_i}$ is homotopy equivalent to $g_i$.
 \end{enumerate}
 
Here, we set $B=G/H$ with vector bundle $-\xi$. We aim to apply Theorem \ref{freed} to show:
\begin{lemma}\label{diffproducttoN}
	$M\times S^{n-1}$ is diffeomorphic to $N$. 
\end{lemma}
\begin{proof}
 Lemma \ref{normalN} (2) and \ref{normalMtimesS} (2) imply that both $\mathfrak{f\circ i}:M\times S^{n-1}\to G/H$ and $f\circ i:N\to G/H$ are normal maps.
 
Recall $\mm{dim}(M)=\mm{dim}(G/H)$, $n-1>\mm{dim}(G/H)$, and 
$$\mm{dim}(M\times S^{n-1})=\mm{dim}(N)=\mm{dim}(G/H)+n-1 \equiv 2\pmod 4$$ 
Thus, $r=\frac{\mm{dim}(N)}{2}$ is odd. Furthermore, $n-1>r>\mm{dim}(G/H)$.

By the cohomology rings \eqref{cohomoN} and \eqref{cohomoMSm-1}, both $\mathfrak{f\circ i}$ and $f\circ i$ induce isomorphisms on $H_k$ for $k<r$, and epimorphisms on $H_r$. Moreover, $B_r(N)=B_r(M\times S^{n-1})$. Since $M\times S^{n-1}, N,$ and $G/H$ are simply connected, both $\mathfrak{f\circ i}$ and $f\circ i$ are $r$-connected.

By Lemma \ref{normalN} (1) and \ref{normalMtimesS} (1), we can construct a normal map $\mathcal F$ from the connected sum $W\# (M\times D^{n})$ to $G/H$ such that $\mathcal F|_{N}=f\circ i$ and $\mc{F}|_{M\times S^{n-1}}=\mathfrak{f\circ i}$. This shows  $\mathfrak{f\circ i}$ and $f\circ i$ are normally bordant, completing the proof. 
\end{proof}

\begin{proof}[Proof of Theorem A]
	By Lemma \ref{constrlemma}, $N$ admits a metric with $\mm{sec}\ge 0$. By Lemma \ref{diffproducttoN}, we complete this proof. 
\end{proof}


\end{document}